\documentclass[sn-mathphys-num]{sn-jnl}

\usepackage{graphicx}%
\usepackage{multirow}%
\usepackage{amsmath,amssymb,amsfonts}%
\usepackage{amsthm}%
\usepackage{mathrsfs}%
\usepackage[title]{appendix}%
\usepackage{xcolor}%
\usepackage{textcomp}%
\usepackage{manyfoot}%
\usepackage{booktabs}%
\usepackage{algorithm}%
\usepackage{algorithmicx}%
\usepackage{algpseudocode}%
\usepackage{listings}%

\usepackage{bm}
\usepackage{a4wide} 

\newtheoremstyle{mythm}
{5pt}
{5pt}
{\itshape}
{}
{\bfseries}
{.}
{.5em}
{}

\newtheoremstyle{mydef}
{5pt}
{5pt}
{}
{}
{\bfseries}
{.}
{.5em}
{}

\theoremstyle{mythm}%
\newtheorem{theorem}{Theorem}
\newtheorem{lemma}{Lemma}
\newtheorem{proposition}{Proposition}%

\theoremstyle{mydef}%
\newtheorem{example}{Example}%
\newtheorem{remark}{Remark}%
\newtheorem{definition}{Definition}%
\newtheorem{assumption}{Assumption}

\renewenvironment{proof}{\smallskip\noindent\emph{\textbf{Proof.}}%
  \hspace{1pt}}{\hspace{-5pt}{\nobreak\quad\nobreak\hfill\nobreak%
    $\square$\vspace{2pt}\par}\smallskip\goodbreak}

\newcommand{\Lip}{\mathrm{Lip}}
\renewcommand{\epsilon}{\varepsilon}
\renewcommand{\phi}{\varphi}
\newcommand{\spt}{\mathop{\rm spt}}

\renewcommand{\d}{\,\mathrm{d}}

\newcommand{\id}{\mathrm{id}}

\def\sign{\mathop{\rm sign}\nolimits}
\def\Sign{\mathop{\rm Sign}\nolimits}

\newcommand{\R}{\mathbb{R}}

\newcommand{\T}{\intercal}

\begin{document}

\title[Optimal Control of Nonlocal Balance Equations]{Optimal Control of Nonlocal Balance Equations}


\author*{\fnm{Nikolay} \sur{Pogodaev}}\email{nickpogo@gmail.com}

\author*{\fnm{Maxim} \sur{Staritsyn}}\email{starmaxmath@gmail.com}

%
%


\abstract{
The paper presents an approach to studying optimal control problems in the space of nonnegative measures with dynamics given by a nonlocal balance law. This approach relies on transforming the balance law into a continuity equation in the space of probabilities, and subsequently into an ODE in a Hilbert space. The main result is a version of Pontryagin's maximum principle for the addressed problem. 
}




\maketitle

\emph{We dedicate this article to the 85th anniversary of Professor Alexander Tolstonogov, whose scientific influence on our professional paths cannot be overstated.}

\section{Introduction\label{sec:level1}}

The paper addresses an optimal control problem for a \emph{nonlocal} balance law~--- the transport equation with a \emph{source} term~--- on the space 
\(
\mathcal{X} \doteq \mathcal{M}^+_c(\mathbb{R}^n)
\)
of nonnegative Borel measures with compact supports in the Euclidean space $\mathbb{R}^n$. Specifically, given the data
\begin{gather*}
I \doteq [0, T], \quad T > 0, \quad U \subset \mathbb{R}^m, \quad \vartheta \in \mathcal{X}, 
\\
\ell \colon \mathcal{X} \to \mathbb{R}, \quad F \colon I \times (U \times \mathcal X \times \mathbb{R}^n ) \to \mathbb{R}^n, \quad \text{and} \quad 
G \colon I \times (U \times \mathcal{X} \times \mathbb{R}^n) \to \mathbb{R},
\end{gather*}
we consider the minimization problem:
\[
(P) \qquad \ell(\mu_T^u) \to \inf, \quad u \in \mathcal{U},
\]
subject to the partial differential equation (PDE):
\begin{align}
\partial_t \mu_t + \nabla_{x} \cdot \big[F_t\big(u_t, \mu_t\big) \, \mu_t\big] = G_t\big(u_t, \mu_t\big) \, \mu_t, \quad t \in I; \quad \mu_0 = \vartheta. \label{eq:u-blaw}
\end{align}

In this formulation:
\begin{itemize}
    \item The set $\mathcal{U} \doteq L^\infty(I; U)$ represents the class of admissible controls $u \colon I \to U$ taking values in a \emph{compact} subset $U$ of $\R^m$. For simplicity, we restrict our attention to controls depending only on time, discussing potential generalizations in concluding sections. This choice places problem $(P)$ within the framework of \emph{ensemble control} \cite{liEnsembleControlFiniteDimensional2011}.
    
    \item The set $\mathcal{X}$ is called the state space. For a given $u \in \mathcal{U}$, a solution $\mu = \mu[u] \colon I \to \mathcal{X}$, $t \mapsto \mu_t$, to \eqref{eq:u-blaw} (in the distributional sense specified below) is referred to as a trajectory corresponding to the control $u$. 

    \item Without loss of generality, we assume that $\vartheta$ is a \emph{probability measure}.  
    All forthcoming results are easily extended to the general case $\vartheta\in \mathcal{X}$ by an appropriate normalization.

    \item The map $F$ is termed a \emph{nonlocal} vector field, reflecting its potential dependence on the distribution $\mu_t$ over the entire base space $\mathbb{R}^n$.
    
    \item The term on the right-hand side of \eqref{eq:u-blaw} is called the \emph{source} (or sink) term. In general \cite{pogodaevNonlocalBalanceEquations2022}, this term could be represented by a controlled operator $\mathcal{X} \to \mathcal{X}$. In our present setting, similar to \cite{Averboukh2024a,Averboukh2024b}, the source term is given by a measure with density relative to $\mu_t$. While the density function $G$ may depend nonlocally on $\mu_t$, the source is always supported on the support $\spt(\mu_t)$ of $\mu_t$. In fact, we have $\spt(\mu_t) = \spt(\widetilde{\mu}_t)$, where $t \mapsto \widetilde{\mu}_t$ solves the continuity equation:
\[
    \partial_t \widetilde{\mu}_t + \nabla_{{x}} \cdot \big[F_t(u_t, \widetilde{\mu}_t) \, \widetilde{\mu}_t\big] = 0, \quad \widetilde{\mu}_0 = \vartheta.
\]
This feature will be essential for our consequent analysis.

\end{itemize}

\subsection{Motivation}

The main motivation for considering dynamical systems of the form \eqref{eq:u-blaw} stems from various applications, where 
the measure $\mu_t$ represents the time-varying distribution of a population of certain ``microscopic'' individuals carrying different ``masses''. When the population is \emph{homogeneous} (i.e., the individuals are homotypic or indistinguishable, having the same mass), and its size remains constant (i.e., $G \equiv 0$), the system is naturally analyzed in a metric space of probability measures. In this conservative case, which has been extensively studied in the literature, $\mu_t$ explicitly represents the probability distribution law of the population, referred to as the \emph{mean field}. This framework, rooted in Statistical Mechanics \cite{cuckerEmergentBehaviorFlocks2007}, provides a robust mathematical foundation for modeling and analyzing the statistical behavior of interacting entities, such as elementary particles, animals, public traffic or opinion agents, see, e.g.,  \cite{carrillo2010particle,cristiani2014multiscale,colombo2011nonlocal,piccoli2011time} and citations therein.

The addressed non-conservative case enables modeling \emph{heterogeneous} populations of agents that differ in their ``influence'' on each other.  Although the range of applications for such models is evidently much broader, and they keep attracting vivid interest  in recent years \cite{Averboukh2024a,Averboukh2024b,CARRILLO20123245,colomboNonLocalSystemsBalance2015,Dll2023}, the related control-theoretical results remain surprisingly limited. In fact, the corresponding \emph{optimal control} theory is not yet developed even at the foundational level. Here, we can only mention our recent works \cite{goncharovaisu,pogodaevNonlocalBalanceEquations2022}. In the first, preliminary results on the well-posedness of a general-form optimal control problem are obtained. In the second, a variant of the classical Pontryagin's maximum principle is derived for the state-linear version of the problem $(P)$.

\subsection{Example: Opinion Dynamics}
\label{ssec:model}

To exemplify the appearance of equation \eqref{eq:u-blaw}, consider the multi-agent model from~\cite{Duteil2022} describing the opinion dynamics in a social group:
\begin{equation}
\label{eq:model}
\begin{cases}
    \dot x_i(t) = \displaystyle\frac{1}{M(t)}\sum_{j=1}^N m_j(t) \, \psi(x_j(t)-x_i(t)),\\
    \dot m_i(t) = \displaystyle\frac{m_i(t)}{M(t)^q}\sum_{j_1=1}^N\cdots \sum_{j_q=1}^N m_{j_1}(t)\cdots m_{j_q}(t) \, S(x_i,x_{j_1},\ldots, x_{j_q}),
\end{cases} \; 1\le i \le N.   
\end{equation}
Each agent in this model is characterized by two parameters: its opinion $x\in \R^n$ and its influence $m\in \R^+$. The ``total influence,'' denoted by $M$, is given by $M(t) = \sum_{i=1}^N m_i(t)$. The functions $\psi \colon \R^n\to \R^n$ and $S\colon (\R^n)^{q+1}\to \R$ describe, respectively, the change in opinion and the change in influence due to interactions among agents. Some meaningful examples of $\psi$ and $S$ can be found in~\cite{Piccoli2019}.

It is shown in~\cite{Duteil2022} that, under mild regularity assumptions on $\psi$ and $S$ and the skew-symmetry condition
\begin{gather}\label{eq:skew}
\exists (i,j)\in \{0,\ldots,q\}\text{ such that } \forall x\in (\R^n)^{q+1} 
\notag\\ 
S(x_0,\ldots,x_i,\ldots,x_j,\ldots,x_q) = -S(x_0,\ldots,x_j,\ldots,x_i,\ldots,x_q),
\end{gather}
the mean-field limit of~\eqref{eq:model} as $N\to+\infty$ is described by the balance equation~\eqref{eq:u-blaw}, where $F$ and $G$ are defined as
\[
F(\mu,x) = \int_{\R^n} \psi(y-x)\d\mu(y),\quad G(\mu,x) = \int_{(\R^n)^q} S(x,x_1,\ldots,x_q)\d\mu(x_1)\cdots\d\mu(x_q).
\]

\subsection{Goals. Main Ideas. Contribution and Novelty}

The main goal of this note is to propose a general approach to the analytical and numerical investigation of the balance law \eqref{eq:u-blaw} and the associated optimization problem $(P)$.

Our approach, which to the best of our knowledge is novel, involves reducing $(P)$ to an equivalent and well-studied ``conservative'' form, represented by the nonlocal continuity equation (without a source term) over an extended base space. This reduction is made possible by the specific structure of the source measure and enables the unique reconstruction of a distributional solution of the original balance law via a barycentric projection. The continuity equation, in turn, can be further expressed as an evolution equation in a Hilbert space, providing access to a broad spectrum of results in the spirit of \cite{tolstonogovDifferentialInclusionsBanach2011}. 

Although this approach is fairly simple and, in essence, reduces the model at hand to an already known case, it nevertheless facilitates the translation of analytical and control-theoretical results, available for the reduced model, to the original one.

As the main result, we demonstrate the application of our methodology to derive the foundational Pontryagin's maximum principle (PMP) \cite{Pontryagin1962} for the problem $(P)$. Notably, this is the first result of its kind for optimal control problems involving \emph{nonlocal balance laws}.

\subsection{Notations and Conventions}

By $\R^n$ we denote the $n$-dimensional Euclidean space equipped with the Euclidean norm $|\cdot|$. 

Let $X$ be a topological space. Its topological dual is denoted by $X^*$, typically endowed with the weak* topology $\sigma(X^*, X)$. 

In finite-dimensional spaces, we distinguish vectors in $\R^n$ from those in its dual space $(\R^n)^*$. Vectors in $\R^n$ are represented as column vectors and indexed by $x^i$, while dual vectors are represented as row vectors and indexed by $p_i$. For a function $\phi \colon \R^n \to \R$, its gradient $\nabla_x \phi(x)$ is always a row vector, while the gradient $\nabla_p \phi(p)$ of $\phi \colon (\R^n)^* \to \R$ is always a column vector. The derivative of a vector field $\phi \colon \R^n \to \R^n$ at $x$ is denoted by $D_x \phi(x)$.

The space ${\mathcal{C}}(X; Y)$ consists of continuous mappings from $X$ to $Y$. If $Y = \R$, we simply write ${\mathcal{C}}(X)$. The space ${\mathcal{C}}_0(\R^n)$ denotes the space of real functions on $\R^n$ that vanish at infinity, equipped with the uniform norm $\|\cdot\|_\infty$. 

The spaces ${\mathcal{C}}_c(\R^n)$ and ${\mathcal{C}}^1_c(\R^n)$
refer to continuous and continuously differentiable
functions with compact supports in $\R^n$, respectively.

The article will frequently feature expressions involving the Lebesgue integral \( \displaystyle\int_{A} \phi(x) \, d\mu(x) \) of various functions over different measure spaces \( (A, \mathcal{A}, \mu) \). The sigma-algebra \( \mathcal{A} \) of subsets of \( A \) is always assumed to be Borel. When no special clarification is needed, we will omit the integration domain and arguments for brevity, assuming that they are understood from the definition of the measure \( \mu \), i.e., we will write simply \( \displaystyle\int \phi \, d\mu \).

By $\mathcal{M}(\R^n)$ we denote the Banach space of finite signed Borel measures on $\R^n$. 
According to the Riesz-Markov representation theorem, $\mathcal{M}(\R^n)$ is isomorphic to $\mathcal{C}_0(\R^n)^*$ with the duality pairing $\langle \cdot, \cdot \rangle \doteq \langle \cdot, \cdot \rangle_{(\mathcal{M}, \mathcal{C}_0)}$  defined as:
\(
\displaystyle\langle \mu, \phi \rangle \doteq \int \phi \, \mathrm{d}\mu.
\)
We equip $\mathcal{M}(\R^n)$ with the Kantorovich-Rubinstein norm:
\[
\|\mu\|_K := \sup \left\{ \int \phi \, \mathrm{d}\mu : \phi \in \mathcal{C}_b(\R^n),\; \|\phi\|_{\text{Lip}} \leq 1 \right\},
\]
where $\|\phi\|_{\text{Lip}} \doteq \max\{\|\phi\|_\infty, \mathrm{Lip}(\phi)\}$, and $\mathrm{Lip}(\phi)$ denotes the minimal Lipschitz constant of $\phi$.

For any Borel mapping $F \colon \R^n \to \R^n$, the pushforward of a measure $\mu \in \mathcal{M}(\R^n)$ under $F$ is denoted by:
\[
F_\sharp \mu \doteq \mu \circ F^{-1}.
\]

By $\mathcal{M}_c(\R^n)$, we denote the set of all compactly supported measures in $\mathcal{M}(\R^n)$. The cone of nonnegative measures and the unit semi-sphere of probability measures in $\mathcal{M}(\R^n)$ are denoted by $\mathcal{M}^+(\R^n)$ and $\mathcal{P}(\R^n)$, respectively. Their subsets of measures with compact supports are denoted by $\mathcal{M}^+_c(\R^n)$ and $\mathcal{P}_c(\R^n)$. The space $\mathcal{P}_c(\R^n)$ is equipped with the $L^2$-Kantorovich distance:
\[
W_2(\rho_1, \rho_2) = \sqrt{\inf_{\Pi} \int_{\R^{2n}} |x-y|^2 \, \mathrm{d}\Pi(x, y)},
\]
where the infimum is taken over all transport plans $\Pi$ between $\rho_1$ and $\rho_2$. A transport plan between $\rho_1$ and $\rho_2$ is a probability measure on $\R^n\times \R^n$ such that $\pi^1_\sharp \Pi = \rho_1$ and $\pi^2_\sharp \Pi = \rho_2$, where $\pi^1,\pi^2\colon \R^n\times\R^n\to \R^n$ are the projections to the factors. We refer to~\cite{bogachevWeakConvergenceMeasures2018} for further details.

The symbol $\mathfrak{L}^n$ represents the Lebesgue measure on $\R^n$, while $\delta_x$ denotes the Dirac measure at $x$. The abbreviations ``a.e.'' and ``a.a.'' mean ``almost everywhere'' and ``almost all'' with respect to  (w.r.t.) $\mathfrak{L}^n$.

The spaces ${L}^1(I; \R^n)$ and ${L}^\infty(I; \R^n)$ consist of measurable Lebesgue-integrable and essentially bounded functions, respectively, defined on the interval $I$ with values in $\R^n$. Their norms are denoted as usual. By $L^2_\mu(\R^n; \R^n)$, we denote the Hilbert space of vector fields, which are square-integrable w.r.t. the measure $\mu$.

When dealing with multivariable functions, we often omit explicit dependence on ``free'' arguments. For example, a function $f = f(x, y, z)$ with a fixed $y$ is written as $f(y)$. Special emphasis is given to the time variable $t$, whose dependence is denoted by a subscript, e.g., $f_t(x, y, \ldots)$.

\section{Reduction to a continuity equation}

We begin with some preliminary facts, specifications, and standing assumptions. 

\subsection{Calculus on the Spaces of Measures}
\label{sec:deriv}

In this section, we recall some necessary elements of analysis and differential calculus in the space $\mathcal M_c(\R^n)$. 

\begin{definition}\label{def:local}
    Let $\mathcal{K}$ be a convex subset of $\mathcal{M}_c(\R^n)$, and 
    $Q \colon \mathcal{K} \to \R$ be a function.
    We say that a property of $Q$ is \emph{local} if it holds for the restriction of $Q$ to $\mathcal{K} \cap \mathcal{M}(E)$, where $E \subset \R^n$ is an arbitrary compact set. Here, $\mathcal{M}(E)$ denotes the set of all signed measures from $\mathcal{M}_c(\R^n)$ supported on $E$.
\end{definition}

For example, $Q$ is \emph{locally bounded} if, for any compact $E \subset \R^n$, there exists $C_E > 0$ such that $|Q(\mu)| \leq C_E$ for all $\mu \in \mathcal{K} \cap \mathcal{M}(E)$. Similarly, $Q$ is \emph{locally Lipschitz} if, for any compact $E \subset \R^m$, there exists $L_E > 0$ such that 
\[
|Q(\mu_1) - Q(\mu_2)| \leq L_E |\mu_1 - \mu_2|_E \quad \text{for all } \mu_1, \mu_2 \in \mathcal{K} \cap \mathcal{M}(E).
\]

\begin{definition}\label{def:diff}
    Let $\mathcal{K}$ be a convex subset of $\mathcal{M}_c(\R^n)$, and let $Q \colon \mathcal{K} \to \mathbb{R}$ be a function. Suppose there exists a continuous and locally bounded functional $\frac{\delta Q}{\delta \mu} \colon \mathcal{K} \times \mathbb{R}^n \to \mathbb{R}$ such that the condition
\begin{equation}
\lim_{t\to 0+}\frac{Q(\mu + t(\mu'-\mu))-Q(\mu)}{t} = \int \frac{\delta Q}{\delta \mu}(\mu,x)\d(\mu'-\mu)(x)\label{deriv}
\end{equation}
    holds for all $\mu, \mu' \in \mathcal{K}$. The map $\frac{\delta Q}{\delta \mu}$ is called the \emph{flat derivative} of $Q$. 

    If the flat derivative $\frac{\delta Q}{\delta \mu}$ is continuously differentiable in $x$ for any $\mu \in \mathcal{K}$, we define another map $\nabla_\mu Q \colon \mathcal{K} \times \mathbb{R}^n \to \mathbb{R}^n$ as 
\[
    \nabla_\mu Q(\mu, x) \doteq \nabla_x \frac{\delta Q}{\delta \mu}(\mu,x).
\]
    The map $\nabla_\mu Q$ is called the \emph{intrinsic derivative} of $Q$.

    We say that $Q$ is \emph{of class $\mathcal{C}^1(\mathcal{K})$} if it admits flat and intrinsic derivatives, both of which are \emph{locally bounded} and \emph{locally Lipschitz}. 
\end{definition}

\begin{remark}~
\begin{enumerate}
    \item Our definitions of flat and intrinsic derivatives mimic similar concepts introduced in \cite{cardaliaguetAnalysisSpaceMeasures2019} and \cite{CardMaster2019} for the cases $\mathcal{K} = \mathcal{P}(\R^n)$ and $\mathcal{K} = \mathcal{P}(\mathbb{T}^n)$, where $\mathbb{T}^n$ denotes the $n$-dimensional torus. In these cases (or, more generally, when $\mathcal{K}$ is such that all its elements $\mu$ share the same total charge $\mu(\R^n)$), the flat derivative is defined up to an additive constant and must satisfy an appropriate normalization condition.

In this paper, we study two classes of $\mathcal{C}^1$-maps: the first corresponds to the choice $\mathcal{K} = \mathcal{M}^+_c(\R^n)$, and the second to $\mathcal{K} = \mathcal{P}_c(\R^n \times \R^+)$. In the first case, since the elements of $\mathcal{K}$ may have different total variations, the original notion of flat derivative \cite{cardaliaguetAnalysisSpaceMeasures2019} is no longer applicable.

\item The presented concepts of derivative naturally extend to vector-valued maps $Q \colon \mathcal{K} \to \mathbb{R}^k$. In this case, the flat derivative $\frac{\delta Q}{\delta \mu}$ takes values in $\mathbb{R}^k$, and the intrinsic derivative, denoted by $D_\mu Q$, takes values in the space of $k \times n$ real matrices.
\end{enumerate}
\end{remark}

\begin{example}
    To illustrate the machinery behind Definition~\ref{def:diff}, we compute the flat and intrinsic derivatives of the functional
\[
\mu \mapsto G(\mu, x) = \iint S(x, x_1, x_2) \, d \mu(x_1) \, d \mu(x_2)
\]
from Section~\ref{ssec:model}.
In this case, we compute:
\begin{align*}
G\left(\mu + t (\mu'-\mu),x\right) - G\left(\mu,x\right) & = t^2\iint S(x,x_1,x_2)\d (\mu'-\mu)(x_1)\d(\mu'-\mu)(x_2)\\
&+t\iint S(x,x_1,x_2)\d \mu(x_1)\d (\mu'-\mu)(x_2)\\
&+t\iint S(x,x_1,x_2)\d \mu(x_2)\d (\mu'-\mu)(x_1)
\end{align*}
and derive:
\begin{align*}
\frac{\delta G}{\delta\mu}(\mu,x,x') 
&=  \int S(x,x_1,x')\d \mu(x_1)+\int S(x,x',x_2)\d \mu(x_2),
\end{align*}
further implying that
\[
\nabla_\mu G(\mu,x) = \nabla_{x'}\frac{\delta G}{\delta\mu}(\mu,x) = \int \nabla_{x'} S(x,x_1,x')\d \mu(x_1)+\int \nabla_{x'} S(x,x',x_2)\d \mu(x_2).
\]
\end{example}

\subsection{Assumptions}

In this section, we introduce the basic assumptions imposed on the functions $F$ and $G$. 

The first list of assumptions addresses their continuity properties.
\begin{assumption}\label{assumption-1}~
\begin{itemize}
  \item The functions $(t,u)\mapsto F_t(u,\mu,x)$ and $(t,u)\mapsto G_t(u,\mu,x)$ are continuous for all $\mu\in \mathcal X$, $x\in \mathbb R^n$;
	
  \item the functions \( (\mu,x)\mapsto F_t(u,\mu,x) \) and $(\mu,x)\mapsto G_t(u,\mu,x)$ are locally Lipschitz, i.e., for any compact $E\subset \R^n$ there exists \( M_E > 0 \) such that
	    \begin{gather*}
        \left|F_t(u,\mu,x) - F_t(u,\mu,x') \right|\le  M_E\left(|x-x'| + \|\mu-\mu'\|_K\right), \\
        \left|G_t(u,\mu,x) - G_t(u,\mu,x') \right|\le  M_E\left(|x-x'| + \|\mu-\mu'\|_K\right),
	  \end{gather*}
		for all \(t\in  I \), \(u \in U\), $\mu,\mu'\in \mathcal M^+(E)$, \(x, x'\in E \);
  \item $F$ and $G$ are sublinear: there exists $C>0$ such that 
   \begin{gather*}
        \left|F_t(u,\mu,x)\right|\le  C(1+\mu(\R^n)), \quad
        \left|G_t(u,\mu,x)\right|\le  C,
	  \end{gather*}
  for all \(t\in  I \), \(u \in U\), $\mu\in \mathcal X$, \(x\in \R^n \).
\end{itemize}
\end{assumption}
The second kit of assumptions is related to the differentiability properties of $F$ and $G$.
\begin{assumption}\label{as-2}~
    \begin{itemize}
    \item The maps $x\mapsto F_t(u,\mu,x)$ and $x\mapsto G_t(u,\mu,x)$ are continuously differentiable for each $t\in I$, $u\in U$ and $\mu\in \mathcal{M}^+(\R^n)$; 
    \item the functions $\mu\mapsto F_t(u,\mu,x)$ and $\mu\mapsto G_t(u,\mu,x)$ are of the class $\mathcal C^1(\mathcal X)$ for each $t\in I$, $u\in U$ and $x\in \R^n$;
    \item the derivatives $D_x F$, $\nabla_xG$, $D_\mu F$, $\nabla_\mu G$ are continuous in $(t,u)$;
    \item the maps $D_x F_t(u,\cdot,\cdot,\cdot)$, $\nabla_xG_t(u,\cdot,\cdot,\cdot)$, $D_\mu F_t(u,\cdot,\cdot,\cdot)$, $\nabla_\mu G_t(u,\cdot,\cdot,\cdot)$ are locally bounded and locally Lipschitz for each $(t,u)\in I\times U$ with some constants independent of $(t,u)$.
\end{itemize}
\end{assumption}

\subsection{State Equation}

Let $\vartheta \in \mathcal X$, and consider the Cauchy problem \eqref{eq:u-blaw}.
This PDE is understood in the sense of distributions, namely, for any test function $\phi \in \mathcal C^1(\R^n)$, the action map $t \mapsto \langle \mu_t, \phi\rangle$ is absolutely continuous as $I \to \R$, and, for a.a. $t \in I$, it holds:
\begin{align}\label{dist}
    \frac{d}{dt} \langle \mu_t, \phi\rangle = \left\langle \mu_t, \nabla_x \phi \cdot F_t(u_t, \mu_t) + G_t(u_t, \mu_t) \, \phi \right\rangle. 
\end{align}
By \cite[Theorem~1]{pogodaevNonlocalBalanceEquations2022}, the regularity assumption \ref{assumption-1} guarantees the existence of a unique distributional solution $\mu_{(\cdot)} =\mu_{(\cdot)}[u]: t \mapsto \mu_t \doteq \mu_t[u]$ to the initial value problem \eqref{eq:u-blaw} under any fixed control $u_{(\cdot)} \in \mathcal U$. 

Building on this result, we now establish a useful ``characteristic'' representation of $\mu$.

\begin{lemma} 
Let $u_{(\cdot)}$ and $\mu_{(\cdot)}$ be defined as above. Introduce a non-autonomous vector field $f\colon I \times \R^n \to \R^n$ and a function $g\colon I \times \R^n \to \R$ by:
\[
    f_t(x) \doteq F_t(u_t, \mu_t, x), \ \ \text{and}\ \  g_t(x) \doteq G_t(u_t, \mu_t, x).
\]

The action of measures $\mu_t$ on test functions $\phi \in \mathcal C^1(\R^n)$ can be expressed as follows: 
\begin{equation}
     \langle \mu_t, \phi\rangle = \left\langle\vartheta, Y_t \, \phi(X_t)\right\rangle, \quad t \in I,\label{chrs}
\end{equation}
where $t \mapsto X_{t}(x)$ 
is a unique solution to the ordinary system:
\[
    \dot X = f_t(X), \quad X_0 = x,
\]
and 
$t \mapsto Y_{t}(x)$ is a unique solution 
$$
Y_{t}(x) \doteq \exp\left(\int_0^{t} g_\tau\left(X_{\tau}(x)\right) \d\tau \right),
$$ 
to the linear scalar Cauchy problem:
\[
    \dot Y = g_t\left(X_{t}{(x})\right) \, Y, \quad Y_0 = 1.
\]
\end{lemma}

\begin{proof} First, notice that $f$ and $g$ are bounded Carath\'{e}odory functions due to hypotheses \ref{assumption-1}, which implies the existence of a unique solution of the characteristic system of ordinary differential equations (ODEs). 

The result is now derived by Leibniz's rule, using the absolute continuity of the action $t \mapsto \langle \mu_t, \phi\rangle$ and the definition \eqref{dist} of a distributional solution to \eqref{eq:u-blaw}:
\begin{align*}
    \frac{d}{dt}\langle \mu_t, \phi\rangle & = \left\langle\vartheta, Y_{t} \{\nabla_x \phi (X_{t}) \cdot f_t(X_{t}) + g_t(X_{t}) \, \phi(X_{t})\}\right\rangle \\
    & = \left\langle \mu_t, \nabla_x \phi \cdot f_t + g_t \, \phi\right\rangle 
    \\
    &\doteq \left\langle \mu_t, \nabla_x \phi \cdot F_t(u_t, \mu_t) + G_t(u_t, \mu_t) \, \phi\right\rangle. 
\end{align*}
\end{proof}

\begin{remark}
This lemma provides a microscopic (agent-wise) interpretation of the distributed (macroscopic) model \eqref{eq:u-blaw}: the measure $\mu_t$
can be viewed as the distribution law for a population of microscopic particles, each characterized by its position $x(t)$ and mass $y(t)$. The position dynamics follow the vector field $f$, and the mass evolves according to the exponential law governed by the function $g$. Note that the formulas \eqref{chrs} are not ``characteristic representations'' in the classical sense of linear PDE theory, as they require prior knowledge of $\mu$.   
\end{remark}

\begin{remark}
When $\vartheta \in \mathcal X$, 
it follows from \eqref{chrs} that all measures $\mu_t$ remain in $\mathcal X$. In other words, the cone of nonnegative measures is invariant under the system \eqref{eq:u-blaw}.
\end{remark}

\subsection{Barycentric Projection}

A brief inspection of the representation \eqref{chrs} reveals a ``barycentric'' dependence of $\mu_t$ on the auxiliary variable $y$, enabling the elimination of the source term in \eqref{eq:u-blaw} by embedding the corresponding dynamics into the ``flow part.''~To formalize this idea, introduce the linear operator of barycentric projection on a component of a product space: 
\begin{equation}\label{beta}
     \beta\colon  \mathcal M_c(\mathbb R^{n+1})\to \mathcal{M}_c(\mathbb R^n), \qquad   \left\langle \beta(\rho), \phi \right\rangle \doteq \int_{\R^{n+1}} y \, \phi(x) \, \mathrm{d}\rho(x, y), \quad \phi \in \mathcal{C}_0(\mathbb R^{n}).
\end{equation}

In general $\beta$ ranges over signed measures. However, it is easy to see that the image of the set
\[
\mathcal Y \doteq \mathcal{P}_c(\R^n\times \R^+)
\] 
belongs to $\mathcal X \doteq \mathcal M_c^+(\R^n)$. In what follows, unless explicitly stated otherwise, we consider $\beta$ as a map from $\mathcal Y$ to $\mathcal X$.

Let us introduce the notation \(\bm x \doteq [x, y]^\T\). By using the operator $\beta$, we define the nonlocal vector field:
\[
     \bm{F}: I \times U \times {\mathcal Y} \times \mathbb{R}^{n+1} \to \mathbb{R}^{n+1}, \ \ \ \
   \bm{F}_t(u, \rho, \bm x) \doteq \begin{bmatrix} F_t(u, \beta(\rho), x) \\ y \, G_t(u, \beta(\rho), x) \end{bmatrix},
\]
and consider the initial value problem for a nonlocal continuity equation on the metric space $(\mathcal Y, W_2)$:
\begin{align}\label{eq:contest}
    \partial_t \rho_t + \nabla_{\bm x} \cdot \left[\bm{F}_t(u_t, \rho_t) \, \rho_t\right] = 0, \quad \rho_0 = \vartheta \otimes \delta_1.
\end{align}
This PDE is understood in the sense of distributions:
\begin{equation}\label{eq:rho_wf}
    \frac{\mathrm{d}}{\mathrm{d}t} \langle \rho_t, \bm \phi \rangle = \left\langle \rho_t, \nabla_{\bm x} \bm \phi \cdot \bm F_t(u_t, \rho_t) \right\rangle,
\end{equation}
for all $\bm \phi \in \mathcal{C}^1(\mathbb{R}^{n}\times \R^+)$, where:
\[
    \nabla_{\bm x} \bm \phi \cdot {\bm F}_t({u}, \rho) \doteq \nabla_{(x, y)} \bm \phi \cdot {\bm F}_t({u}, \rho) \doteq  \nabla_{x} \bm \phi \cdot F_t({u}, \beta(\rho), x) + y \, \partial_{y} \bm \phi \, G_t({u}, \beta(\rho), x).
\]

We will see in the following section that $\beta\colon \mathcal{Y} \to \mathcal{X}$ is locally Lipschitz (in the sense of Definition~\ref{def:local}) when $\mathcal{Y}$ and $\mathcal{X}$ are furnished, respectively, with the $L^2$-Kantorovich distance $W_2$ and the Kantorovich-Rubsenstein distance, generated by $\|\cdot\|_K$. As a consequence, $\bm{F}$ satisfies the standard regularity assumptions \cite{piccoliTransportEquationNonlocal2013}, ensuring the well-posedness of the Cauchy problem \eqref{eq:contest}.

We now show that there is one-to-one correspondence between solutions of the continuity equation~\eqref{eq:contest} and the original balance law \eqref{eq:u-blaw}.

\begin{theorem}\label{murhopro}
Suppose that assumptions \emph{\ref{assumption-1}} are fulfilled. Then, the unique solutions $t \mapsto \mu_t$ and $t \mapsto \rho_t$  to the Cauchy problems \eqref{eq:u-blaw} and \eqref{eq:contest}, respectively, are related as follows:
\begin{equation}
    \mu_t = \beta(\rho_t), \quad \forall t \in I.\label{beta-mu}
\end{equation}
\end{theorem}
\begin{proof}
    By restricting the definition \eqref{eq:rho_wf}  to test functions of the structure $\bm \phi(x, y) = y \, \phi(x)$, $\phi \in \mathcal C^1(\R^n)$, 
    we obtain:
\begin{align*}
    \frac{d}{dt} \int_{\R^{n+1}} y \, \phi(x) \d \rho_t(x, y) & \doteq \frac{d}{dt} \int_{\R^{n}\times \R^+} \bm \phi \d \rho_t\\ & = \left\langle \rho_t, \nabla_{(x,y)} \bm \phi \cdot \bm F_t\left(u_t, \rho_t\right) \right\rangle\\
    & = \int_{\R^{n}\times \R^+} y \, \nabla_x \phi({x}) \cdot F_t\left({u}_t, \beta(\rho_t), x\right) \d {\rho_t}(x, y)\\ 
    & + \int_{\R^{n}\times \R^+} y \, \phi(x) \, G_t\left(u_t,\beta(\rho_t), x\right) \d{\rho_t}(x,y).
\end{align*}
Denoting
$\pi_t \doteq \beta(\rho_t)$,
we rewrite:
\begin{align*}
    \frac{d}{dt} \langle \pi_t, \phi \rangle & = \int_{\R^n} \nabla_x \phi \cdot F_t(u_t, \pi_t) \d \pi_t + \int_{\R^n} \phi \, G_t(u_t, \pi_t) \d{\pi_t}\\
    & = \big\langle \pi_t, G_t(u_t, \pi_t)\, \phi\big\rangle
    + \big\langle \pi_t, \nabla_x \phi\cdot G_t(u_t,\pi_t)\big\rangle.
\end{align*}
The uniqueness of a solution to the balance law then implies $\pi_t = \mu_t$ for all $t \in I$, which completes the proof.
\end{proof}

\subsection{Properties of the Barycentric Projection}

We now establish some useful properties of the operator $\beta$. First, let us demonstrate that $\beta$ preserves Lipschitz continuity.

\begin{lemma}\label{lem:betalip}
For each $b\ge 1$, the map \(\beta \colon\mathcal P_c(\R^n \times [0,b]) \to \mathcal M^+(\R^n)\) 
is $2b$-Lipschitz.
\end{lemma}
\begin{proof}
    Fix two probability measures $\rho,\rho'\in \mathcal P_c(\R^n\times [0,b])$ and denote by $\Pi$ an optimal transport plan between $\rho$ and $\rho'$ in the sense \cite{ambrosioLecturesOptimalTransport2021}.
    For any test function $\phi \in C_b(\R^n)$ such that $\|\phi\|_{\Lip}\le 1$, we have 
    \begin{align*}
        \left|\left\langle \beta(\rho),\phi\right\rangle - \left\langle \beta(\rho'),\phi\right\rangle \right| 
        &= 
        \left|\int y \phi(x) \d\rho'(x,y)- \int y' \phi(x') \d\rho'(x',y') \right|\\
        &= 
        \left|\int \left(y \phi(x) -  y' \phi(x')\right) \d \Pi(x,y,x',y')\right| \\
        &\le 
        \int \left|y \phi(x) -  y' \phi(x')\right| \d \Pi(x,y,x',y')\\
        &\le 
        \int |y| \left|\phi(x) -  \phi(x')\right| \d \Pi(x,y,x',y') \\ 
        &+
        \int |\phi(x)| \left|y_1-y'\right| \d \Pi(x,y,x',y')\\
         &\le 
        \int |y| \left|x- x'\right| \d \Pi(x,y,x',y') \\ 
        &+
        \int \left|y-y'\right| \d \Pi(x,y,x',y').
    \end{align*}
    We apply H\"older's inequality to the latter integrals and obtain
     \begin{align*}
        \left|\left\langle \beta(\rho_1),\phi\right\rangle - \left\langle \beta(\rho_2),\phi\right\rangle \right| 
         &\le 
        \left(\int |y|^2 \d \Pi(x,y,x',y')\right)^{1/2}\left(\int |x-x'|^2 \d \Pi(x,y,x',y')\right)^{1/2} \\ 
        &+
        \left(\int 1^2 \d \Pi(x,y,x',y')\right)^{1/2}\left(\int (y-y')^2 \d \Pi(x,y,x',y')\right)^{1/2}\\
        &\le 2b \left(\int \left(|x-x'|^2 +(y-y')^2\right) \d \Pi(x,y,x',y')\right)^{1/2}.
    \end{align*}
    Since the transport plan $\Pi$ is optimal, the right-hand side can be written as $2b W_2(\rho,\rho')$. It remains to note that the inequality 
    \[
    \left|\left\langle \beta(\rho),\phi\right\rangle - \left\langle \beta(\rho'),\phi\right\rangle \right| \le 2b W_2(\rho,\rho')
    \]
    holds for any continuous test function $\phi$ with $\|\phi\|_{\Lip}\le 1$. This implies, by the very definition of the Kantorovich-Rubinstein norm, that $\|\beta(\rho)-\beta(\rho')\|_K\le 2b W_2(\rho,\rho')$, completing the proof.
\end{proof}

Since the operator $\beta$ maps the state space ${\mathcal{Y}}$ of the reduced system \eqref{eq:contest} into the original state space $\mathcal{X}$, it allows transferring any functional $Q \colon \mathcal{X} \to \R$ to a functional $\widehat{Q}\colon\mathcal{Y}\to \R$ using the pullback operation:
\[
\widehat{Q}(\rho) \doteq Q(\beta(\rho)), \quad \rho \in {\mathcal Y}.
\]
Lemma~\ref{lem:betalip} immediately implies the following result.
\begin{lemma}\label{lem:qlip}
    If $Q \colon \mathcal{X} \to \R$ is locally Lipschitz, then $\widehat{Q} \colon \mathcal{Y} \to \R$ is also locally Lipschitz.
\end{lemma}

Finally, we show that the pullback operation preserves the property of differentiability.
\begin{lemma}\label{lem:qdiff}
    If $Q \in \mathcal{C}^1\left(\mathcal{X}\right)$, then $\widehat{Q} \in \mathcal{C}^1({\mathcal Y})$. Moreover, the flat and intrinsic derivatives of these functions are related as follows:
    \begin{equation}
    \frac{\delta\widehat{Q}}{\delta \rho}(\rho, \bm x) = y \frac{\delta Q}{\delta \mu}(\beta(\rho),x), \;
    \nabla_{\rho} \widehat{Q}(\rho, \bm x) \doteq \nabla_\mu Q(\beta(\rho),x,y) =
    \begin{bmatrix}
        y \nabla_\mu Q(\beta(\rho),x) & \; \frac{\delta Q}{\delta \mu}(\beta(\rho),x)
    \end{bmatrix}.
    \label{diff_beta}
    \end{equation}
\end{lemma}
\begin{proof}
It follows from $Q \in \mathcal{C}^1\left(\mathcal{X}\right)$ that
\begin{align*}
\lim_{t\to 0+}\frac{\widehat{Q}(\rho + t(\rho'-\rho)) - \widehat{Q}(\rho)}{t} 
    &= \lim_{t\to 0+}\frac{Q\big(\beta(\rho)+t(\beta(\rho')-\beta(\rho))\big) - Q(\beta(\rho))}{t} \\
    &= \int \frac{\delta Q}{\delta \mu}\left(\beta(\rho), x\right)\d \left(\beta(\rho')-\beta(\rho)\right)(x) \\
    &= \int y \frac{\delta Q}{\delta \mu}\left(\beta\left(\rho\right), x\right)\d \left(\rho'-\rho\right)(x,y).
\end{align*}
Therefore, $\frac{\delta \widehat{Q}}{\delta \rho}$ exists and is given by
\[
\frac{\delta \widehat{Q}}{\delta \rho}(\rho, \bm x) = y \frac{\delta Q}{\delta \mu}(\beta(\rho), x).
\]
As a consequence, the differentiability of $x \mapsto \frac{\delta Q}{\delta \mu}(\mu,x)$ implies the differentiability of $\bm x \mapsto \frac{\delta \widehat{Q}}{\delta \rho}(\rho, \bm x)$. One can easily compute that
\begin{equation*}
\nabla_{\rho} \widehat{Q}(\rho, \bm x) = \nabla_{\bm x} \frac{\delta \widehat{Q}}{\delta \rho}(\rho, \bm x) = 
\begin{bmatrix}
    y \nabla_\mu Q(\beta(\rho),x) &  \frac{\delta Q}{\delta \mu}(\beta(\rho),x)
\end{bmatrix}.
\end{equation*}
Now, as the representation formulas \eqref{diff_beta} have been established, we can easily verify that the derivatives $\frac{\delta \widehat{Q}}{\delta \rho}$ and $\nabla_\rho \widehat{Q}$ are locally Lipschitz (with the aid of Lemma~\ref{lem:betalip}) and locally bounded. Hence, we conclude that $\widehat{Q} \in \mathcal{C}^1({\mathcal{Y}})$.
\end{proof}

\subsection{Problem Transformation}

Let the data $T$, $\ell$, $F$, $G$, and $\mathcal{U}$ be as in the statement of $(P)$. 
Consider the following optimal control problem in the metric space $({\mathcal{Y}}, W_2)$ of probability measures:
\[
(\bm{P}) \qquad 
\widehat{\ell}({\rho}_T) \doteq \ell(\beta({\rho}_T)) \to \inf, \quad u \in \mathcal{U},
\]
subject to the nonlocal continuity equation \eqref{eq:contest} on ${\mathcal{Y}}$.

Due to Theorem~\ref{murhopro}, $(\bm{P})$ is equivalent to the original problem $(P)$, meaning that both problems share the same optimal controls (if they exist). The optimal trajectories $\mu$ of $(P)$ are recovered from the optimal states ${\rho}$ in $(\bm{P})$ via the barycentric projection mapping $\beta$.

\subsection{Pontryagin's Maximum Principle}

This section presents the main result of this paper: a version of Pontryagin's Maximum Principle (PMP) for the problem $(P)$. The result will be presented in two equivalent forms.

The first version is derived by translating the necessary optimality condition for the reduced problem $(\bm{P})$, as recalled from \cite{bonnetPontryaginMaximumPrinciple2019,bonnetNecessaryOptimalityConditions2021}. This follows the tradition of control theory on manifolds and involves a Hamiltonian equation in the cotangent bundle $\mathbb{CB} \doteq (\R^{n} \times \R^+) \times (\R^{n} \times \R^{+})^* \simeq (\R^{n} \times \R^+)^2$ of the corresponding base space.

The second version of the PMP mimics its conventional formulation \cite{Pontryagin1962}. It involves decomposing the Hamiltonian equation into the corresponding primal and adjoint components, similar to \cite{chertovskihOptimalControlNonlocal2023}. In this case, the adjoint component is represented by the barycentric projection of the Hamiltonian measure onto the subspace of co-vectors.

\begin{proposition}[Pontryagin's Maximum Principle for problem $(\bm {P})$ \cite{bonnetPontryaginMaximumPrinciple2019,bonnetNecessaryOptimalityConditions2021}]
    Let the hypotheses \emph{\ref{as-2}} hold together with \emph{\ref{assumption-1}}, and let $t \mapsto (u_t, \rho_t \doteq \rho_t[u])$ be an optimal process for the problem $(\bm P)$. 
    Then, the following condition is satisfied for a.a. $t \in I$:
\[
    \bm H\left(u_t, \gamma_t\right)=\max_{u\in U}\bm H(u, \gamma_t).
\]
Here, 
\begin{align*}
\bm{H}_t(u,\gamma) = \int \bm p \cdot \bm F_t(u, \pi^1_\sharp \gamma, \bm x)\, \d \gamma(\bm{x},\bm{p})
\end{align*}
is a lifted Hamiltonian, and $\gamma \in \mathcal{C}\left(I; \mathcal{P}_c\left(\mathbb{CB}\right)\right)$ is a solution to the Hamiltonian system:
\begin{equation}
\begin{cases}
  \partial_t \gamma_t + \nabla_{(\bm x, \bm p)} \cdot \left( \mathbb{J}_{2(n+1)}\, D_{\gamma} \bm H_t\left(u_t, \gamma_t\right) \, \gamma_t \right) = 0, \\
  \pi^1_\sharp \gamma_t = \rho_t, \quad \text{for all } t \in I, \\
  \pi^2_\sharp \gamma_T = (-\nabla_{\rho} \widehat{\ell}(\rho_T))_\sharp \, \rho_T,
\end{cases}
\label{eq:hamil1}
\end{equation}
where $\bm p \doteq [p, q]$ is the co-vector of $\bm x$; $\pi^{1,2} : \R^{2(n+1)} \to \R^{n+1}$, with $\pi^1(\bm x, \bm p) \doteq \bm x$ and $\pi^2(\bm x, \bm p) \doteq \bm p$, are the projection maps onto the corresponding subspaces, and $\mathbb{J}_{2(n+1)}$ is the symplectic matrix
\[
\mathbb{J}_{2(n+1)} = \begin{pmatrix}
    0 & \id \\
    -\id & 0
\end{pmatrix}
\]
of dimension $2(n+1)$.
\end{proposition}

\begin{remark}
    The term ${\mathbb J}_{2(n+1)} D_{\gamma} \bm H\left(u, \gamma\right)$ can be explicitly computed as follows:
\[
{\mathbb J}_{2(n+1)} D_{\gamma} \bm H_t\left(u, \gamma, \bm x, \bm p\right) = \begin{pmatrix}
    \bm F_t(u, \pi^1_\sharp \gamma, \bm x) \\
    -\bm p \, \displaystyle D_{\bm x} \bm F_t(u, \pi^1_\sharp \gamma, \bm x) - \int \bm p' D_{\rho} \bm F_t(u,  \pi^1_\sharp \gamma, \bm x, \bm x') \, \d \gamma(\bm x', \bm p')
\end{pmatrix}.
\]
\end{remark}

We now translate this result to the terms of the original problem $(P)$ using the relation \eqref{beta-mu}.  
\begin{theorem}[PMP for problem $({P})$: Version I]\label{pmp}
    Let hypotheses \emph{\ref{assumption-1}} and \emph{\ref{as-2}} hold, and $t \mapsto (u_t, \mu_{t} = \mu_t[u])$ be an optimal process for the problem $(P)$. Then, the following condition holds at a.e. $t \in I$:
\[
    H\left(u_t, \gamma_t\right)=\max_{u\in U}H(u, \gamma_t),
\]
where
\begin{align}
 H_t(u,\gamma) \doteq \int p \cdot F_t\left(u, \beta\left(\pi^1_\sharp \gamma\right),  x\right)\d \gamma(x, y, p, q)+ \int q \, y \, G_t\left(u, \beta\left(\pi^1_\sharp \gamma\right), x\right)\d \gamma(x, y, p, q),\label{pf}
\end{align}
and $\gamma\in \mathcal C\left(I;\mathcal P_c\big(\mathbb C\mathbb B\big)\right)$ satisfies the following system
\begin{equation}
\begin{cases}
  \partial_t\gamma_t + \nabla_{(x, y, p, q)}\cdot\left(\vec H_t\left(u_t, \gamma_t\right) \, \gamma_t\right)=0,\\
    \pi^1_\sharp\gamma_0=\vartheta\otimes \delta_1,\quad \text{ for all }\quad t \in I,\\
  \pi^2_\sharp\gamma_T = \left([x \ y] \mapsto -\begin{bmatrix}
    y\, {\nabla_{\mu}}\ell(\mu_T,x) &  \frac{\delta \ell}{\delta \mu}(\mu_T,x)
\end{bmatrix}\right)_\sharp \pi^1_\sharp\gamma_T,
\end{cases}\label{eq:hamil}
\end{equation}
and the function $\vec H_t=[\vec H_t^1, \ldots, \vec H_t^4]$ is defined as:
\begin{align*}
\vec H_t^1(u, \gamma, x) \doteq & \  F_t(\beta(\pi^1_\sharp\gamma), x),
\\
\vec H_t^2(u, \gamma, x, y) \doteq & \ y \, G_t(\beta(\pi^1_\sharp\gamma), x)
\\
   \vec H_t^3(u, \gamma, x, y) \doteq &   - p\, D_xF_t(\beta(\pi^1_\sharp\gamma), x) 
   - \displaystyle\int y'p'\, D_{\mu}F_t(\beta(\pi^1_\sharp\gamma),x,x')\d \gamma(x',y',p',q')\\
   \ &- q  y\, \nabla_xG_t(u, \beta(\pi^1_\sharp\gamma),x)
     - \displaystyle\int y'q'y  \, \nabla_{\mu}G_t(\beta(\pi^1_\sharp\gamma), {x}, x')\d \gamma(x',y',p',q')\\
  \vec H_t^4(u,\gamma, x, y)  \doteq  &  - q \, G_t\left(\beta(\pi^1_\sharp\gamma),x\right) - \displaystyle\int p'\cdot \frac{\delta F_t}{\delta \mu} (u,\beta(\pi^1_\sharp),x, x') d \gamma(x',y',p',q')\\
  \ & - \displaystyle\int q'y \, \frac{\delta G_t}{\delta \mu}(u, \beta(\pi^1_\sharp\gamma_t), x, x')\d \gamma(x',y',p',q'),
\end{align*}
Moreover,
\[
  \beta(\pi^1_\sharp\gamma_t)=\mu_t,\quad \text{ for all }\quad t \in I.
\]
\end{theorem}
\begin{proof}
We start by observing that    
$
\bm{H}_t(\cdot,\cdot) = H_t(\cdot,\cdot),
$
which follows directly from the definition of $\bm{F}$. Thus, the result reduces to showing that the solution of \eqref{eq:hamil} coincides with the solution of \eqref{eq:hamil1}. 

The existence and uniqueness of solutions are established by \cite{MR4097258}. What remains is merely to demonstrate that the conditions of \eqref{eq:hamil} and  \eqref{eq:hamil1} are identical. To this end, we rewrite all derivatives in \eqref{eq:hamil1} using the projected variable $\mu \doteq \beta(\rho)$. First, applying \eqref{diff_beta}, we find:
\[
\nabla_{\rho} \widehat{\ell}(\rho) = 
\begin{bmatrix}
y \, \nabla_\mu \ell(\beta(\rho), x) & \frac{\delta \ell}{\delta \mu} (\beta(\rho), x)
\end{bmatrix}
\doteq
\begin{bmatrix}
y \, \nabla_\mu \ell(\mu, x) & \frac{\delta \ell}{\delta \mu} (\mu, x)
\end{bmatrix}.
\]

Next, we compute:
\[
D_{\bm x} \bm{F}_t(u, \rho, x, y) = 
D_{(x, y)} \begin{bmatrix}
F_t(u, \beta(\rho), x) \\ 
y \, G_t(u, \beta(\rho), x)
\end{bmatrix} = 
\begin{bmatrix}
D_x F_t(u, \beta(\rho),x) & 0_n \\ 
y \nabla_x G_t(u, \beta(\rho), x) & G_t(u, \beta(\rho), u, x)
\end{bmatrix},
\]
which implies:
\[
\bm{p} \, D_{\bm{x}} \bm{F}_t(u, \rho, \bm{x}) = 
\begin{bmatrix}
p \cdot D_x F_t(u, \beta(\rho), x) + q y\, D_x G_t(u, \beta(\rho), x) &\hspace{5pt} 
q \, G_t\left(u, \beta(\rho), x\right)
\end{bmatrix}.
\]

Introducing:
\[
\mathcal{F}_t(u, \mu, \bm{x}) = 
\begin{bmatrix}
F_t(u, \mu, x) \\ 
y\, G_t(u, \mu, x)
\end{bmatrix},
\]
we compute:
\[
\frac{\delta \mathcal{F}_t(u, \mu, \bm{x}, \bm x')}{\delta \mu} = 
\begin{bmatrix}
\frac{\delta F_t}{\delta \mu} (u, \mu, x, x') \\[0.1cm] 
y \, \frac{\delta G_t}{\delta \mu} (u, \mu, x, x')
\end{bmatrix} \qquad \Rightarrow \ \ \frac{\delta \bm{F}_t(u, \rho, \bm{x}, \bm{x}')}{\delta \rho} = y'\frac{\delta \mathcal{F}_t(u, \beta(\rho), \bm{x}, \bm x')}{\delta \mu},
\]
and, consequently,
\begin{equation}\label{diff}
D_{\rho} \bm{F}_t(u, \rho, \bm{x}, \bm{x}') = 
\begin{bmatrix}
y'\, D_\mu F_t (u, \beta(\rho), x, x') & \frac{\delta F_t}{\delta \mu} (u, \beta(\rho), x, x') \vspace{5pt}\\
y'y \, \nabla_\mu G_t(u, \beta(\rho), x, x') & y \, \frac{\delta G_t}{\delta \mu} (u, \beta(\rho), x, x')
\end{bmatrix}.
\end{equation}

Substituting these into the representation and recalling that $\mu=\beta(\rho) = \beta(\pi^1_\sharp\gamma)$, we obtain:
\begin{multline*}
\bm{p}' D_{\rho} \bm{F}_t(u, \rho, \bm{x}, \bm{x}') \\= 
\begin{bmatrix}
y'\, p'\, D_\mu F_t(u, \mu, x, x') + y'\, q'\, y \, \nabla_\mu G_t(u, \mu,x, x') &\hspace{8pt} 
p'\cdot \frac{\delta F_t}{\delta \mu} (u,\mu,x, x') +~q' \, y \; \frac{\delta G_t}{\delta \mu} (u,\mu,x, x')
\end{bmatrix}.
\end{multline*}

Putting everything together, we confirm that the systems \eqref{eq:hamil} and \eqref{eq:hamil1} are indeed equivalent reformulations, completing the proof.
\end{proof}

We now pass to another (equivalent) representation of the same necessary optimality condition.

Introduce a vector measure $\nu = [\nu_i : 1 \leq i \leq n+1]$, $\nu_i \in {\mathcal Y}$, via the action:
\[
\langle \nu_i, \bm{\phi}\rangle \doteq \int \bm{p}_i \, \bm{\phi}(\bm{x}) \, \mathrm{d}\gamma(\bm{x}, \bm{p}), \quad \bm{\phi} \in \mathcal{C}^1(\mathbb{R}^n \times \mathbb{R}^{+}).
\]
In \cite{chertovskihOptimalControlNonlocal2023}, it is shown that the map $t \mapsto \nu_t$ is a unique distributional solution to the following system of balance laws:
\begin{align}
  \label{eq:adjoint1}
  \partial_{t}\nu_i + \nabla_{\bm{x}}\cdot(\bm{f} \nu_i) &= \sum_{j=1}^{n+1} \left\{\left(\int [\bm{m}_t(\bm{x}', \cdot)]^j_i \, \mathrm{d}\nu_j(\bm{x}')\right)\rho - \partial_{\bm{x}^i}\bm{f}^j \, \nu_j\right\}, \quad 1\leq i\leq n+1,
\end{align}
where the dependence on $t$ is omitted for brevity, and the following abbreviations are used:
\[
\bm{f}_t(\bm{x}) \doteq \bm{F}_t\left(u_t, \rho_t, \bm{x}\right), \quad \bm{m}_t(\bm{x}', \bm{x}) \doteq D_{\rho} \bm{F}_{t}\left(u_t, \rho_t, \bm{x}', \bm{x}\right).
\]

Moreover, the measure $\nu_T$ is absolutely continuous w.r.t. $\rho_T$, with the Radon-Nikodym derivative:
\begin{equation}
  \label{eq:adj_term}
  -\frac{\mathrm{d}\nu_T}{\mathrm{d}\rho_T}(x, y) = \nabla_{\rho}\widehat{\ell}(\rho_T, x, y) \doteq \begin{bmatrix}
    y \, \nabla_{\mu}\ell(\mu_T, x) &  \frac{\delta \ell}{\delta \mu}(\mu_T, x)
  \end{bmatrix}.
\end{equation}

Let us denote:
\[
{\psi_i} \doteq ([x \ y] \to x)_\sharp \nu_i, \quad 1\leq i \leq n, \quad \text{and} \quad \xi \doteq \beta(\nu_{n+1}),
\]
and observe that the Hamiltonian \eqref{pf} can be expressed in terms of $\psi_i$ and $\xi$. Indeed, the first term reads:
\[
\int p \cdot F\left(u, \beta\left(\pi^1_\sharp \gamma\right), x\right) \, \mathrm{d}\gamma(x, y, p, q) \doteq \sum_{i=1}^n\int [F(u, \mu, x)]^i \, \mathrm{d}\nu_i(x, y) \doteq \sum_{i=1}^n \left\langle \psi_i, F^i\right\rangle,
\]
and the second term is restated as
\[
\int q \, y \, G\left(u, \beta\left(\pi^1_\sharp \gamma\right), x\right) \, \mathrm{d}\gamma(x, y, p, q) \doteq \int G \left(u, \mu, x\right) \, \mathrm{d}\beta(\nu_{n+1})(x) \doteq \big\langle \xi, G\big\rangle.
\]

The terminal condition \eqref{eq:adj_term} can also be rewritten in terms of $(\mu, \psi, \xi)$:
\begin{equation}\label{psiT}
  - \frac{\mathrm{d} (\psi_i)_{T}}{\mathrm{d} \mu_T} = [\nabla_\mu \ell(\mu_T)]_i,  
\end{equation}
and, analogously:
\begin{equation}\label{xiT}
  - \frac{\mathrm{d} \xi_{T}}{\mathrm{d} \mu_T} = \frac{\delta \ell}{\delta \mu}(\mu_T).
\end{equation}

Rewriting the $i$th balance law in \eqref{eq:adjoint1}, for $1 \leq i \leq n$, in the distributional form with the restricted class of test functions of the form $\bm{\phi}(x, y) = \phi(x)$, we get:

\begin{align*}
\frac{\d}{\d t} \left\langle \psi_i, \phi \right\rangle  \doteq \frac{\d}{\d t} \left\langle \nu_i, \bm{\phi} \right\rangle 
&= \int \nabla_x \phi(x) \cdot F(u, \mu, x) \d \nu_i\\ 
&\quad + \sum_{j=1}^n \int d \nu_j(x', y') \int y \, \phi(x) \, [D_\mu F(u, \mu, x', x)]^j_i  \d \rho(x, y)\\
&\quad + \int d \nu_{n+1}(x', y') \, y' \int y \, \phi(x) \, \nabla_\mu G_i(u, \mu, x', x) \d \rho(x, y)\\
&\quad - \sum_{j=1}^n \int \phi(x) \, [D_x F(u, \mu, x)]^j_i \d \nu_j(x, y)\\
&\quad - \int y \, \phi(x) \, \nabla_x G_i(u, \mu, x) \d \nu_{n+1}(x, y)\\
&\doteq \int \nabla_x \phi(x) \cdot F(u, \mu, x) \d \psi_i(x)\\ 
&\quad + \sum_{j=1}^n \int \d \psi_j(x') \int \phi(x) \, [D_\mu F(u, \mu, x', x)]^j_i \d \mu(x)\\
&\quad + \int \d \xi(x') \int \phi(x) \, \nabla_\mu G_i(u, \mu, x', x) \d \mu(x)\\
&\quad - \sum_{j=1}^n \int \phi(x) \, [D_x F(u, \mu, x)]^j_i \d \psi_j(x)\\
&\quad - \int \phi(x) \, \nabla_x G_i(u, \mu, x) \d \xi(x).
\end{align*}
The result is the distributional form of the following PDE:
\begin{align}
\partial_t \psi_i + \nabla_{x} \cdot (F(u, \mu) \psi_i) 
&= \sum_{j=1}^n \left\{ \left( \int \left[D_\mu F(u, \mu, x', \cdot)\right]^j_i \d \psi_j(x') \right) \mu - [D_x F(u, \mu)]^j_i \psi_j \right\} \notag\\
\label{psj}
&\quad + \left( \int \nabla_\mu G_i(u, \mu, x', \cdot) \d \xi(x') \right) \mu - \nabla_x G_i(u, \mu) \, \xi.
\end{align}

Similarly, using test functions of the form $\bm{\phi}(x, y) = y \, \phi(x)$ in the representation of a distributional solution to the last equation in \eqref{eq:adjoint1}, we derive:
\begin{align*}
\frac{\d}{\d t} \left\langle \xi, \phi \right\rangle & \doteq\frac{\d}{\d t} \left\langle \nu_{n+1}, \bm{\phi} \right\rangle\\
& \quad = \int y\, \nabla_x\phi(x) \cdot F(u, \mu, x) \, \d \nu_{n+1}(x, y)\\
& \quad + \int \phi(x) \, y \, G(u, \mu, x) \, \d \nu_{n+1}(x, y)\\
&\quad + \sum_{j=1}^n \int \d \nu_j(x', y') \int y \, \phi(x) \, \left[ \frac{\delta F(u, \mu, x', x)}{\delta \mu} \right]^j \d \rho(x, y)\\
&\quad + \int \d \nu_{n+1}(x', y') \, y' \int y \, \phi(x) \, \frac{\delta G(u, \mu, x', x)}{\delta \mu} \d \rho(x, y)\\
&\quad - \int y \, \phi(x) \, G(u, \mu, x) \d \nu_{n+1}(x, y)\\
&\doteq \int \nabla_x\phi(x) \cdot F(u, \mu, x) \, \d \xi(x)\\
&\quad + \sum_{j=1}^n \int \d \psi_j(x') \int \phi(x) \, \left[ \frac{\delta F(u, \mu, x', x)}{\delta \mu} \right]^j \d \mu(x)\\
&\quad + \int \d \xi(x') \int \phi(x) \, \frac{\delta G(u, \mu, x', x)}{\delta \mu} \d \mu(x),
\end{align*}
with the strong form:
\begin{equation}
\partial_t \xi + \nabla_x\cdot(F(u, \mu)\, \xi) = \left\{ \sum_{j=1}^n \int \left[ \frac{\delta F(u, \mu, x', \cdot)}{\delta \mu} \right]^j d \psi_j(x') + \int \frac{\delta G(u, \mu, x', \cdot)}{\delta \mu} \d \xi(x') \right\} \mu.
\label{xi}
\end{equation}

The backward system~\eqref{psiT}--\eqref{xi} consisting of linear nonlocal balance laws is referred to as the \emph{adjoint} of \eqref{eq:u-blaw}; its well-posedness is established similarly to \cite[Section 4.2]{chertovskihOptimalControlNonlocal2023}. 

Finally, we arrive at the following result:
\begin{theorem}[PMP for problem $({P})$: Version II]
\label{vII}
Let the process $t\mapsto (u_t, \mu_t)$ satisfy all assumptions of Theorem~\ref{pmp}. Then, the following pointwise condition holds at a.a. $t \in I$: 
\[
    \mathcal{H}(u_t, \psi_t, \xi_t) = \max_{u \in U} \mathcal{H}(u, \psi_t, \xi_t),
\]
where
\[
    \mathcal{H}_t(u, \psi, \xi) \doteq \sum_{i=1}^{n} \left\langle \psi_i, F^i_t(u, \mu) \right\rangle + \left\langle \xi, G_t(u, \mu) \right\rangle,
\]
and $t \mapsto ((\psi_1)_t, \ldots, (\psi_n)_t, \xi_t)$ is a backward solution to the adjoint system~\eqref{psiT}--\eqref{xi}.
\end{theorem}

Although the latter assertion is merely a reformulation of Theorem~\ref{pmp}, it is much more convenient for the practical analysis of concrete examples. Moreover, this reformulation is vital for the development of PMP-based indirect numerical algorithms for optimal control. One of the key features here is that, in contrast to the measure $\gamma_t$, which is supported at the moment $T$ on the graph of a function and is therefore always singular w.r.t. $\mathfrak{L}^{2(n+1)}$ (even if $\rho_T$ is absolutely continuous w.r.t. $\mathfrak{L}^{(n+1)}$), the measures $(\psi_i)_t$ and $\xi_t$ inherit the absolute continuity of $\mu_t$ w.r.t. $\mathfrak{L}^n$, enabling the use of standard algorithms for numerical integration of hyperbolic PDEs. We refer to \cite{chertovskihOptimalControlNonlocal2023} for further discussions of this topic.

\begin{example}
   To illustrate the modus operandi of Theorem~\ref{vII}, consider the following example from~\cite{goncharovaisu}:
\[
\mathcal{I}[u] = \ell(\mu_1) \doteq \frac{1}{2} \int x^2 \, \d\mu_1(x) \to \min, \quad |u_t| \leq 1,
\]
subject to the balance equation:
\[
\partial_t \mu + u_t \, \partial_x \mu = -u_t \mu, \quad t \in [0,1], \quad \mu_0 = \vartheta \in \mathcal{M}^+_c(\R).
\]

Theorem~\ref{vII} provides the structure of the optimal controls as:  
\[
u_t \in \arg\max_{|u|\leq 1} \mathcal{H}(u, \psi_t, \xi_t) \doteq \arg\max_{|u|\leq 1} u \int \d(\psi_t - \xi_t) = \Sign(\psi_t - \xi_t)(\R),
\]
where the set-valued map $\Sign$ is defined as: 
\[
\Sign a = \frac{a}{|a|} \text{ for } a \neq 0, \quad \Sign 0 = [-1,1].
\]
The adjoint trajectories $\psi$ and $\xi$ satisfy the same transport PDEs:
\[
\partial_t \psi + u_t \, \partial_x \psi = 0, \quad \partial_t \xi + u_t \, \partial_x \xi = 0,
\]
with terminal conditions:
\[
[\psi_1 \ \xi_1] \doteq -\big[\nabla_\mu \ell(\mu_1) \ \frac{\delta \ell}{\delta \mu}(\mu_1)\big] = -[x \ x^2/2] \ \mu_1.
\]

Since the adjoint equations lack source terms, it follows that:
\[
[\psi_t(\R) \ \ \xi_t(\R)] \equiv [\psi_1(\R) \ \ \xi_1(\R)],
\]
and the maximum condition implies the structure of the optimal controls:
\[
u_t \in \Sign \int \big(x^2/2 - x\big) \, \d\mu_1(x).
\]

By using the characteristic representation \eqref{chrs}, we express:
\[
\mu_t = e^{-a(t)} (\tau_{a(t)})_\sharp \vartheta, \quad \tau_a(x) \doteq x + a, \quad a(t) = a[u](t) \doteq \int_0^t u_s \, \d s.
\]
Thus, an optimal control is determined as a solution to the operator inclusion:
\begin{equation}
\label{eq:max}
u_t\in 
 \Sign\int\big(x^2/2 - x\big)\d (\tau_{a[u](1)})_\sharp\vartheta = \Sign\left\{ \int \big(x + a[u](1) -1\big)^2\d\vartheta(x) -\vartheta(\mathbb R)\right\}.
\end{equation}

Remark that if $\vartheta = \delta_{x_0}$ with $x_0 \in \R$, then the original optimal control problem~$(\bm P)$ is equivalent to the problem
\begin{gather*}
\min \frac{1}{2} X_1^2 Y_1, \quad 
\dot{X}_t = u_t, \ \, X_0 = x_0, \quad 
\dot{Y}_t = -u_t Y_t, \ \, Y_0 = 1, \quad |u_t| \leq 1.
\end{gather*}
One can show that the optimality condition for the latter problem provided by the classical PMP coincides with~\eqref{eq:max}, where $\vartheta=\delta_{x_0}$.

Let us analyze the condition \eqref{eq:max}. First, notice that its right-hand side is independent of \( t \). Hence, $u_t$ can differ from $\pm 1$ only if 
\[
    \int \big(x + a[u](1) -1\big)^2\d\vartheta(x) =\vartheta(\mathbb R).
\]
Any $u$ satisfying this condition is a singular extremal.

Suppose that
\[
    \int \big(x + a[u](1) -1\big)^2\d\vartheta(x) > \vartheta(\mathbb R).
\]
Then, $u \equiv 1$, and $a(1) = 1$, leading to the condition:
\[
\int x^2 \, \d \vartheta  > \vartheta(\R).
\]
Finally, the case \( u \equiv -1 \) is provided by the condition
\[
\int (x-2)^2 \, \d \vartheta < \vartheta(\R). 
\]

When
\[
\int (x-2)^2 \, \d \vartheta < \vartheta(\R) < \int x^2 \, \d \vartheta,
\]
both $u \equiv \pm 1$ are extremal,
and their optimality is verified by comparing the corresponding cost values: 
\[
\mathcal{I}[u \equiv 1] = \frac{1}{2e} \int (x+1)^2 \, \d \vartheta, \quad 
\mathcal{I}[u \equiv -1] = \frac{e}{2} \int (x-1)^2 \, \d \vartheta,
\]
yielding the relation:
\[
u = \sign \int \left(e^2(x-1)^2 - (x+1)^2\right) \, \d \vartheta(x) 
= \sign \int \left\{x^2 + 2\frac{1+e^2}{1-e^2}x + 1\right\} \, \d \vartheta.
\]

\end{example}

\section{Reduction to an Evolution Equation in a Hilbert space}\label{tolston}

We have shown in Theorem~\ref{murhopro} that the balance equation~\eqref{eq:u-blaw} can be reduced to the continuity equation~\eqref{eq:contest}. On the other hand, it is known that the solution of~\eqref{eq:contest} can be represented as $\rho_t = \Phi_{t\sharp}(\vartheta \otimes \delta_1)$, where $\Phi\colon I\times \R^{n+1}\to \R^{n+1}$ is a unique solution to the Cauchy problem 
\[
\frac{d}{dt} \Phi_{t}(\bm x) = \bm F_t\left(u_t,\Phi_{t\sharp}(\vartheta\otimes \delta_1),\Phi_{t}(\bm x)\right), \quad \Phi_{0}(\bm x) = \bm x.
\]
The probability measure $\rho_t$ being an element of $\mathcal Y$ is also an element of the larger space $\mathcal{P}_2(\R^{n+1})$~--- the space of probability measures with finite second moments $\int|\bm x|^2\d\rho(\bm x)<+\infty$. Therefore, the quantity
\[
\int |\bm x|^2\d\rho_t(\bm x) = \int |\Phi_t(\bm x)|^2\d(\vartheta\otimes \delta_1)(\bm x)
\]
is bounded for any $t\in I$, which implies that $\Phi_t\in L^2_{\vartheta\otimes \delta_1}(\R^{n+1};\R^{n+1})$ for all $t\in I$.

Consider the Hilbert space $\mathbb H = L^2_{\vartheta\otimes \delta_1}(\R^{n+1};\R^{n+1})$ and define a function $\mathcal G\colon I\times U\times \mathbb H\to \mathbb H$ by 
\[
\mathcal G_t(\upsilon,\Psi) \doteq \bm F_t\left(\upsilon,\Psi_{\sharp}(\vartheta\otimes \delta_1),\Psi\right),\quad t\in I,\; \upsilon\in U,\; \Psi\in \mathbb H.
\]
Remark that $\mathcal G_t(\upsilon,\Psi)\in \mathbb H$ because $\bm F_t$ is locally bounded (thanks to~\ref{assumption-1}) and $\vartheta$ is compactly supported. 

It is obvious that $\mathcal G$ is measurable in $t$ and continuous in $\upsilon$. Let us study its regularity with respect to $\Psi$.

\begin{lemma}\label{lem:Hlip}
    Let $\rho \mapsto \bm F_t(\upsilon, \rho, \bm x)$ be $M$-Lipschitz for all $t\in I$, $\upsilon \in U$ and $\bm x\in \R^{n+1}$ with some $M>0$.
    Then the map $\Psi \mapsto \mathcal{G}_t(\upsilon,\Psi)$ is  $2M$-Lipschitz for all $t\in I$ and $\upsilon\in U$.
\end{lemma}
\begin{proof}
First, note that
    \begin{align*}
        \|\mathcal{G}_t(\upsilon,\Psi^1) - \mathcal{G}_t(\upsilon,\Psi^2)\|^2_{\mathbb H} \!=\! \int \left|\bm F_t(\upsilon, \Psi^1_\sharp(\vartheta\otimes \delta_1),\Psi^1(\bm x))-
        \bm F_t(\upsilon, \Psi^2_\sharp(\vartheta\otimes \delta_1),\Psi^2(\bm x))
        \right|^2\!\d (\vartheta\otimes \delta_1)(\bm x)
    \end{align*}
Since
\begin{align*}
    \left|\bm F_t(\upsilon, \Psi^1_\sharp(\vartheta\otimes \delta_1),\Psi^1(\bm x))\!-\!
        \bm F_t(\upsilon, ,\Psi^2(\bm x))
        \right|&\le M\left(W_2\left(\Psi^1_\sharp(\vartheta\otimes \delta_1),\Psi^2_\sharp(\vartheta\otimes \delta_1)\right) + \left|\Psi^1(\bm x)\!-\!\Psi^2(\bm x)\right|\right)\\
        &\le M\left(\left\|\Psi^1-\Psi^2\right\|_{\mathbb H} + \left|\Psi^1(\bm x)-\Psi^2(\bm x)\right|\right),
\end{align*}
we conclude that 
\begin{align*}
        \|\mathcal{G}_t(\upsilon,\Psi^1) - \mathcal{G}_t(\upsilon,\Psi^2)\|^2_{\mathbb H} \le M^2 \left(
        2\left\|\Psi^1\!-\!\Psi^2\right\|_{\mathbb H}^2 + 2\left\|\Psi^1\!-\!\Psi^2\right\|_{\mathbb H}\int \left|\Psi^1(\bm x)\!-\!\Psi^2(\bm x)\right|\!\d(\vartheta\otimes\delta_1)(\bm x)
        \right).    
\end{align*}
    Using H\"older's inequality gives 
    \[
    \|\mathcal{G}_t(\upsilon,\Psi^1) - \mathcal{G}_t(\upsilon,\Psi^2)\|_{\mathbb H}^2\le 4M^2 \left\|\Psi^1-\Psi^2\right\|_{\mathbb H}^2,
    \]
    as desired.
\end{proof}

There is still a problem: we cannot expect that $\bm F$ is globally Lipschitz in $\rho$ in our setting. In fact, Assumption~\ref{assumption-1} and Lemma~\ref{lem:betalip} imply that it is merely \emph{locally} Lipschitz. However, one may easily show that all trajectories of the differential inclusion 
\begin{equation}\label{eq:incl}
\dot \Psi_t \in \mathcal{G}_t(U,\Psi_t),\quad \Psi_0 = \id,    
\end{equation}
satisfy the following property: there exists a compact set $E\subset \R^{n+1}$ such that
\[
\spt\left((\Psi_{t})_\sharp (\vartheta\otimes \delta_1)\right)\subset E \quad \forall t\in I.
\]
This means, thanks to~\ref{assumption-1} and Lemma~\ref{lem:Hlip}, that $\Psi\mapsto\mathcal G_t(\upsilon,\Psi)$ is globally Lipschitz when restricted to the solution set of~\eqref{eq:incl}.

The final observation enables the direct translation of certain classical results for Lipschitz ODEs in Hilbert spaces~\cite{tolstonogovDifferentialInclusionsBanach2011}, such as Filippov-Wazewski's Theorem and the Bang-Bang principle, to balance equations of the form~\eqref{eq:u-blaw}. However, we do not explore this direction further in the present paper. 

\section{Conclusion}

The paper explores an optimal control problem for a nonlocal balance equation with a specific ``semi-linear" source term. Our key idea consists is embedding this problem into the well-studied class of control problems on the space of probabilities, which, as a main byproduct, grants an adequate version of Pontryagin's maximum principle.

All presented results are easily generalized to the case when the space $\mathcal{U}$ of ordinary controls is replaced by generalized (sliding mode) controls of the Young-Warga-Gamkrelidze type \cite{young1969,warga1972optimal,Gamkrelidze1978}. Moreover, we believe that the proposed approach is applicable for translating mean-field versions of the PMP \cite{bonnetPontryaginMaximumPrinciple2019,averboukhPontryaginMaximumPrinciple2022} within the class of controls depending on the state variable $x$.

A natural direction for future work is the development of optimization algorithms for the numerical solution of the stated problem, similar to \cite{chertovskihOptimalControlNonlocal2023}, as well as the derivation of feedback optimality principles in the spirit of \cite{goncharovaisu}.

Finally, in Section~\ref{tolston}, we propose a further reduction of the control system~\eqref{eq:u-blaw} to a differential inclusion in a Hilbert space. This reduction potentially provides a pathway for studying the topological structure of the solution set of~\eqref{eq:u-blaw} by leveraging the well-established properties of the corresponding differential inclusion.

%
%


\bibliography{references.bib}

\end{document}